\documentclass[reqno]{amsart}
\usepackage{amssymb}  
\usepackage[T2A]{fontenc}
\usepackage[utf8]{inputenc}
\usepackage[english]{babel} 
\usepackage{enumitem}\usepackage{cite}
\usepackage{xcolor}
\usepackage{breakurl}
\usepackage{hyperref} 
\hypersetup{
    colorlinks=true,
    linkcolor=teal,  
    urlcolor=cyan,
    citecolor=blue
    }
\newtheorem{thm}{Theorem}
\newtheorem{lem}{Lemma}
\newtheorem{prop}{Proposition}

\begin{document}

\title[Conjugate generation of sporadic almost simple groups]{Conjugate generation of sporadic\\ almost simple groups}
\author{Danila O. Revin}%
\address{Danila O. Revin
\newline\indent Sobolev Institute of Mathematics,
\newline\indent 4, Koptyug av.
\newline\indent 630090, Novosibirsk, Russia
} 
\email{revin@math.nsc.ru
\newline\indent ORCID: \href{https://orcid.org/0000-0002-8601-0706}{0000-0002-8601-0706}}

\author{Andrei V. Zavarnitsine}%
\address{Andrei V. Zavarnitsine
\newline\indent Sobolev Institute of Mathematics,
\newline\indent 4, Koptyug av.
\newline\indent 630090, Novosibirsk, Russia
} 
\email{zav@math.nsc.ru
\newline\indent ORCID: \href{https://orcid.org/0000-0003-1983-3304}{0000-0003-1983-3304}
}
\thanks{This research was carried out within the State Contract of the Sobolev Institute of Mathematics (FWNF-2022-0002). The first author is partially supported by the National Natural Science Foundation of China, grant~\#12371021}
\maketitle
\begin{quote}
\noindent{\sc Abstract. } 
As defined by Guralnick and Saxl, given a nonabelian simple group $S$ and its nonidentity automorphism $x$, a natural number $\alpha^{\phantom{S}}_{S}(x)$ is the minimum number of conjugates of $x$ in $\langle x,S\rangle$ that generate a subgroup containing~$S$. In this paper, for every sporadic group $S$ other than the Monster and an automorphism $x$ of $S$ of prime order,   we complete the determination of the precise value of~$\alpha^{\phantom{S}}_{S}(x)$.
\medskip

 \end{quote}

\section{Introduction}

Let $S$ be a nonabelian simple group (which we identify with $\operatorname{Inn}S$) and let  $x\in \operatorname{Aut}S$ be a nonidentity automorphism (possibly, inner). R.\,Guralnick and J.\,Saxl~\cite{GS} introduced the notation   $\alpha(x)=\alpha^{\phantom{S}}_{S}(x)$ for the minimum number of elements conjugate to $x$ in $G=\langle x,S\rangle$ that generate~$G$. Thus, $\alpha^{\phantom{S}}_{S}(x)$ has a constant value on every conjugacy class $x^S=x^G$ and so may be viewed as a numeric invariant of the class. Even before~\cite{GS} appeared, J.\,Moori \cite{M,M1,M2} called $\alpha^{\phantom{S}}_{S}(x)$ \emph{the rank}\footnote{We note that Moori and some other authors, e.\,g. in \cite{A,A2,AI,AI1,AI2,AIK,AK,AM,BMS,B,BM,BM1,BM2,IA,M,M1,M2}, use the term ``rank'' in the case when $x$ induces on $S$ an inner automorphism.} of class $x^S$. 

Given an arbitrary finite simple group $S$, Guralnick and Saxl find a number $m=m(S)$ such that $\alpha^{\phantom{S}}_{S}(x)\leqslant m$ for every automorphism $x\ne 1$ of $S$. In the case of alternating and classical groups, these estimates are precise. They can be achieved on transpositions for the alternating groups and on transvections or reflections for the classical groups. The classes that are not transpositions, transvections, or reflections are supplied in~\cite{GS} with their own upper bounds. However, these bounds and the estimates for the exceptional and sporadic groups are not always precise. Several relevant remarks and conjectures are formulated in~\cite{GS}. There is also a lengthy list of papers \cite{A,A2,AI,AI1,AI2,AIK,AK,AM,BMS,B,BM,BM1,BM2,IA,M,M1,M2,DMPZ,RZ1,RZ2, WGR}, in which not only the estimates by Guralnick and Saxl are improved, but also some precise bounds are found on the ranks of many conjugacy classes for various simple groups. 

A significant number of these papers go back to the pioneering work by Moori \cite{M,M1,M2} who proved as early as 1993--1994 \cite{M,M1} that if $S=F_{22}$ and ${x\in S}$~is a $3$-transposition then $\alpha^{\phantom{S}}_{S}(x)\in\{5,6\}$.  Subsequently, J.\,Hall and L.\,Soicher \cite[Theorems (1.1)--(1.3)]{HS} classified the groups of $3$-transpositions that may be generated by at most five $3$-transpositions. In particular, it was shown that the Fischer groups $Fi_{22}$ and $Fi_{23}$ cannot be generated by five $3$-transpositions, which confirmed Moori's conjecture \cite{M1} that $\alpha^{\phantom{S}}_{S}(x)=6$ when $S=Fi_{22}$ and $x$~is a $3$-transposition. With his paper \cite{M2}, Moori initiated a systematic calculation of the ranks of inner conjugacy classes of sporadic simple groups. In a series of subsequent articles  by J.\,Moori, F.\,Ali, A.\,B.\,M.\,Basheer, M.\,A.\,F.\,Ibrahim, M.\,A.\,Al-Kadhi, et al. \cite{A,A2,AI,AI1,AI2,AIK,AK,AM,BMS,B,BM,BM1,BM2}, the ranks of inner conjugacy classes were found for many sporadic groups and the Tits group. 

A major breakthrough for sporadic groups was made by L.\,Di Mar\-tino, M.\,A.\,Pel\-legrini, and  A.\,E.\,Zalesskii in \cite{DMPZ}. Their result reduces the determination to a limited number of cases, and, for most inner classes, it even gives the exact value of the rank. 
More precisely, if $S$ is not the Monster $M$ and $x$ is a nonidentity element of $S$ then either \cite[Theorem~3.1]{DMPZ} gives a precise value of $\alpha^{\phantom{S}}_{S}(x)$ or one of the following cases\footnote{As in \cite{DMPZ}, slightly abusing the notation we will often denote a conjugacy class and its representative by the same symbol.} holds:
\begin{itemize}
\item[$\bullet$] $(S,x)=(Fi_{22},3B)$ and $\alpha^{\phantom{S}}_{S}(x)\in\{2,3\}$,
\item[$\bullet$] $(S,x)=(Suz,3A)$ and $\alpha^{\phantom{S}}_{S}(x)\in\{3,4\}$,
\item[$\bullet$]  $(S,x)\in\{(Fi_{22},2A),(Fi_{23},2A)\}$ and $\alpha^{\phantom{S}}_{S}(x)\in\{5,6\}$. 
\end{itemize}     

Also, for $S=M$, \cite[Theorem~3.1]{DMPZ} states that $\alpha^{\phantom{S}}_{S}(x)\in\{2,3\}$ if the order of $x$ is bigger than $2$, and $\alpha^{\phantom{S}}_{S}(x)\in\{3,4\}$ if $x$ has order~$2$.

Although Di\,Martino, Pellegrini, and Zalesskii have left some freedom for the possible values of the ranks of certain inner classes, the parameters
$\alpha^{\phantom{S}}_{S}(x)$ for $x\in S$ have virtually been completely determined by now, except for $S=M$. Thus, in the Fischer groups $F_{22}$ and $Fi_{23}$, the conjugacy classes denoted by $2A$ contain $3$-transpositions, and the above-mentioned result by Hall and Soicher \cite{HS} implies that $\alpha^{\phantom{S}}_{S}(x)=6$ whenever $(S,x)\in\{(Fi_{22},2A),(Fi_{23},2A)\}$. F.\,Ali \cite[Theorem~11]{A} showed that $\alpha^{\phantom{S}}_{S}(x)=3$ for $(S,x)=(Fi_{22},3B)$. We will also prove the following assertion:

\begin{prop}\label{suz3a} If $S=Suz$ and  $x$ is in class $3A$ of $S$ then $\alpha^{\phantom{S}}_S(x)=4$.
\end{prop}

Proposition~\ref{suz3a} corrects the inaccuracy in the survey papers \cite[Assertion 1 in Section~3.1]{BM} and \cite[Assertion (1) in Section~3]{BM1}, where it is stated, among other things, that $\alpha^{\phantom{S}}_{S}(x)=3$ for $(S,x)=(Suz,3A)$. Neither of \cite{BM,BM1} cites a source of this information.

In summary, the following theorem holds.
\begin{thm}\label{inn_aut} Let $S$~be a sporadic group, other than the Monster, and let $x\in S\setminus\{1\}$. Then 
\begin{itemize}
    \item[]either $|x|> 2$ and  $\alpha^{\phantom{S}}_{S}(x)=2$, except in the following cases:
    \begin{itemize}
         \item[$\bullet$] $(S,x)\in\{(J_2,3A)$, $(HS,4A)$, $(McL,3A)$, $(Ly,3A)$, $(Co_1,3A)$,\\ $(F_{22},3A)$, $(Fi_{22},3B)$,
         $(Fi_{23},3A)$, $(Fi_{23},3B)$, $({Fi_{24}}',3A)$, $({Fi_{24}}',3B)\}$ and $\alpha^{\phantom{S}}_{S}(x)=3$,
     \item[$\bullet$] $(S,x)=(Suz,3A)$ and $\alpha^{\phantom{S}}_{S}(x)=4$,
    \end{itemize}
    \item[]or $|x|= 2$ and $\alpha^{\phantom{S}}_{S}(x)=3$, except in the following cases: 
    \begin{itemize}
              \item[$\bullet$]  $(S,x)\in\{(J_2,2A),(Co_2,2A),(B,2A)\}$ and $\alpha^{\phantom{S}}_{S}(x)=4$, 
         \item[$\bullet$]  $(S,x)\in\{(Fi_{22},2A),(Fi_{23},2A)\}$ and $\alpha^{\phantom{S}}_{S}(x)=6$. 
    \end{itemize}
\end{itemize}
\end{thm}

Therefore, as far as the inner conjugacy classes of the sporadic groups $S$, there are some open questions only for the Monster. For the ``outer'' conjugacy classes, i.\,e. those contained in $\operatorname{Aut} S\setminus S$, the information on the values of  $\alpha^{\phantom{S}}_{S}(x)$ has until recently been exhausted by the general estimates of Guralnick and Saxl which, in the case $S={Fi_{24}}'$, for example, gives $\alpha^{\phantom{S}}_{S}(x)\leqslant 8$.

We recall that $\operatorname{Aut} S\ne S$ precisely for the following twelve sporadic groups~$S$:
\begin{equation}\label{list}
    M_{12},M_{22},J_2,J_3,McL,O'N,HS,He,Suz,HN,Fi_{22},{{Fi_{24}}'}. 
\end{equation}
In our recent paper \cite{RZ2}, it was shown that, for $x\in\operatorname{Aut} S\setminus S$, we either have $\alpha^{\phantom{S}}_{S}(x)\leqslant 4$, or $(S,x)=({{Fi_{24}}'},2C)$ and $\alpha^{\phantom{S}}_{S}(x)=5$. The peculiarity of class $2C$ in $\operatorname{Aut} ({Fi_{24}}')=Fi_{24}$ is that it is a class of $3$-transpositions. The fact that $Fi_{24}$ is generated by five $3$-transpositions was established by S.\,Norton \cite{N}, and that $Fi_{24}$ cannot be generated by four $3$-transpositions was proved in \cite[Theorem~(1.1)]{HS}.

As remarked in~\cite{GS}, the values of $\alpha^{\phantom{S}}_{S}(x)$ with $x$ of prime order are of special interest. Indeed, for $x\in\operatorname{Aut} S$ and $y\in\langle x\rangle$, if $\langle y^{g_1},\dots, y^{g_m}\rangle\geqslant S$ then $\langle x^{g_1},\dots, x^{g_m}\rangle\geqslant S$. Therefore, $\alpha^{\phantom{S}}_{S}(x)\leqslant\alpha^{\phantom{S}}_{S}(y)$. Note that an element $y\in\langle x\rangle$ may be chosen so that $|y|$ is prime. 

This work continues \cite{RZ1,RZ2}, where we refine the estimates on $\alpha^{\phantom{S}}_{S}(x)$ for $S$ sporadic and $^2F_4(q^2)'$, where $x$ has prime order. Besides proving Proposition \ref{suz3a}, the aim of the present paper is to indicate, for every sporadic group $S$ with $\operatorname{Aut} S\ne S$ and every $x\in\operatorname{Aut} S\setminus S$ of prime order, a precise value of $\alpha^{\phantom{S}}_{S}(x)$. Since  $|\operatorname{Aut} S:S|= 2$ for the groups in~(\ref{list}), every such $x$ is an involution. Our result is as follows.

\begin{thm}\label{main} Let $S$ be a sporadic group and let $x\in \operatorname{Aut}S\setminus S$ be of order $2$. Then $\alpha^{\phantom{S}}_{S}(x)=3$ except 
in the following cases:
\begin{itemize}
    \item[$(i)$] $(S,x)\in\big\{(M_{22},2B), (HS,2C), (Fi_{22},2D)\big\}$ and $\alpha^{\phantom{S}}_{S}(x)=4$;
    \item[$(ii)$] $(S,x)=({Fi_{24}}',2C)$ and $\alpha^{\phantom{S}}_{S}(x)=5$.
\end{itemize}
\end{thm}

We note that, in item $(i)$, the group $\operatorname{Aut}S$ is a {\em $4$-transposition group} with respect to the class $D=x^S$, i.\,e. $\operatorname{Aut}S$ is generated by the involutions in $D$ and the product of every two such involutions has order at most $4$. Other known examples of sporadic $4$-transposition groups 
$S$ with  $\alpha^{\phantom{S}}_{S}(x)=4$ are $Co_2$ and $B$, where $x\in D=2A$ in both cases, see \cite[Theorem 3.1]{DMPZ}.

\medskip {\em Remark.\/}
After the original version of this preprint had been published, we received a letter from Prof. J.\,M\"uller dated May 15, 2025 mentioning the paper \cite{fmbw} of 2019, where the values of $\alpha^{\phantom{S}}_{S}(x)$ were determined independently using other methods for all sporadic groups $S$ and all nonidentity $x\in \operatorname{Aut}S$. Theorems~\ref{inn_aut} and~\ref{main} of the present paper are particular cases of the results in \cite{fmbw} whose existence we had had no knowledge of. We are thankful to Prof. M\"uller for this information and for pointing out an inaccurate citation of \cite{DMPZ} in our original statement of Theorem~\ref{inn_aut}.

\section{Preliminaries}

We consider finite groups only.

The following lemma is known as Brauer's trick, see \cite{Brauer}.

\begin{lem}\label{uf}
Let $A$ and $B$ be subgroups of a group $G$. If there is a non-principal ordinary irreducible character $\chi$ of $G$ such that
$$
 ( \chi_A, 1_A ) + ( \chi_B, 1_B ) > ( \chi_{A\cap B}, 1_{A\cap B} )
$$
then $\langle A, B \rangle$ is  a proper subgroup of $G$.
\end{lem}

It is known from character theory that,
given elements $a,b$ and $c$ of a group $G$, the number $\operatorname{m}(a,b,c)$ of pairs
$(u,v)$, where $u$ is conjugate to $a$, $v$ is conjugate to $b$, and
$uv=c$, can be found from the character table using the formula
$$
\mathrm{m}(a,b,c)=\frac{|G|}{|\operatorname{C}_G(a)||\operatorname{C}_G(b)|}\sum\limits_{\chi\in\mathrm{Irr}(G)}\frac{\chi(a)\chi(b)\overline{\chi(c)}}{\chi(1)},
$$
see \cite[Exercise~(3.9), p.~45]{Isaacs}. For our purposes, the value of $\mathrm{m}(a,b,c)$ can thus be determined either from the Atlas \cite{atlas} or using the computer algebra system \texttt{GAP}~\cite{GAP} which has the relevant built-in function  \texttt{ClassMultiplicationCoefficient()}.

\begin{lem}\label{GuKa}
Let $S$ be a sporadic simple group and let $nX$ be a conjugacy class of $G=\operatorname{Aut} S$.
\begin{itemize}
\item[$(i)$] Assume that $p$ is a prime such that $(S,nX,p)$ is one of 
\begin{multline*}
(M_{12},2C,11),(M_{22},2C,11),(J_2,2C,7),(J_3,2B,19),(O'N,2B,31),\\ 
(HS,2D,11),(HN,2C,19).
\end{multline*}
Then
\begin{itemize}
    \item a Sylow $p$-subgroup of $S$ is cyclic of order~$p$;
     \item for every nonidentity $g\in G$, there exists $s\in S$ of order~$p$ such that $S\leqslant \langle g,s\rangle$;
     \item $\operatorname{m}(nX,nX,pA)>0$.
\end{itemize}
In particular, 
if $x\in nX$ then $\alpha^{\phantom{S}}_S(x)\leqslant 3$. 

\item[$(ii)$] Assume that $(S,nX)$ is one of
$$(Fi_{22},16AB),({Fi_{24}}',29AB),(He,14CD),(Suz,14A).$$ 
Then, for every $g\in G\setminus S$, there exists $x\in nX$ such that $G=\langle g,x\rangle$. 
\end{itemize}
\end{lem}

\begin{proof} $(i)$ The existence of $s$ as stated is proven in \cite[Proposition~6.2]{GK}.
This implies that if $x\in nX$ then $S\leqslant \langle x,s\rangle$ 
for some $s\in S$ of order~$p$. Since $\langle s\rangle$ is a Sylow $p$-subgroup of $S$, we may assume that $s\in pA$. 
Using \cite{GAP} we check that $\operatorname{m}(nX,nX,pA)>0$
in all cases under consideration, see \cite{RZ3} for a detailed \texttt{GAP} code. This implies that $s=x_1x_2$ for some $x_1,x_2\in nX$. Therefore, $S\leqslant \langle x,x_1,x_2 \rangle$ and so $\alpha^{\phantom{S}}_S(x)\leqslant 3$. 

$(ii)$ See \cite[Section~4.8, Table~9]{BGK}.
\end{proof}

\section{Proof of main results}\label{proof}

We begin by proving Proposition~{\rm\ref{suz3a}} which is the final clarifying step in the proof of Theorem \ref{inn_aut}.

\begin{proof}[Proof of Proposition~{\rm\ref{suz3a}}.] Conjugacy class $3A$ of $S$ has the following property. 
Every pair of elements in $3A$ generates a subgroup isomorphic to one of the groups
\begin{equation}\label{5types}
\mathbb{Z}_3,\ \mathbb{Z}_3\times \mathbb{Z}_3,\ A_4,\ A_5,\ \text{or}\ SL_2(3).
\end{equation}
This fact is known and can also be checked using \cite{GAP}. A detailed \texttt{GAP} code for such a verification is available in \cite{RZ3}. We briefly outline here the idea behind it. $S$ can be constructed from two explicit {\em standard} generators $a$ and $b$, where $a$ is in class $2B$, $b$ is in class $3B$, $ab$ has order $13$ and $(ab)^2b$ has order $15$, see \cite{AtlRep}. We check that $t=((ab^2)^3ab)^7$ has order $3$
and the centralizer of $t$ in $S$ has order $9797760$, i.\,e. $t$ is in class $3A$. This enables us to construct this class using $t$ as a representative. 
$S$ acts by conjugation on pairs of elements in $3A$.
It suffices to consider the subgroups of $S$ generated by orbit representatives under this action. Note that these orbits are naturally in a one-to-one correspondence with the orbits of $C_S(t)$ on class $3A$. We find that there are $8$ such orbits and, if we exclude the ones containing inverses, there are only $5$ orbits whose representatives generate precisely the $5$ types of subgroups (\ref{5types}).

The finite simple groups generated by a conjugacy class $D$ of elements of order~$3$ such that every two members of $D$ generates one of the groups (\ref{5types}) were classified in~ \cite{S}. The structure of subgroups generated by three elements in $D$ was clarified in \cite[Section 1]{S}, whence it follows that $Suz$ cannot be such a subgroup. In particular, $\alpha^{\phantom{S}}_S(x)>3$ and so $\alpha^{\phantom{S}}_S(x)=4$ by \cite[Theorem 3.1]{DMPZ}.

\end{proof}

Taking into account the information about conjugacy classes from the introduction and the Atlas \cite{atlas}, 
Theorem~\ref{main} amounts to the following assertion which we are going to prove.

\begin{prop}\label{main1} Let $S$ be a sporadic group and let $x$
be a representative of conjugacy class $nX$ of $G=\operatorname{Aut} S$.
\begin{itemize}
    \item[$(i)$] If $(S,nX)$ is one of
    \begin{multline*}
    (M_{12},2C),(M_{22},2C),(J_{2},2C),(J_{3},2B),(McL,2B),(O'N,2B),(HS,2D),\\
    (He,2C),(Suz,2C),(Suz,2D),(HN,2C),(Fi_{22},2E),(Fi_{22},2F),({Fi_{24}}',2D)
    \end{multline*}
    then $\alpha^{\phantom{S}}_{S}(x)=3$;
    \item[$(ii)$] If $(S,nX)$ is $(M_{22},2B)$, $(HS,2C)$, or $(Fi_{22},2D)$ then 
    $\alpha^{\phantom{S}}_{S}(x)=4$.
\end{itemize}
\end{prop}

\begin{proof} Since $x$ has order $2$, we have $\alpha^{\phantom{S}}_S(x)\geqslant 3$, because two involutions always generate a solvable subgroup. In view of Lemma~\ref{GuKa}$(i)$, it remains to consider the cases listed below. A \texttt{GAP} code for calculating the relevant constants $\operatorname{m}(a,b,c)$ is available in~\cite{RZ3}. We will use without further reference the information from~\cite{atlas} on the maximal subgroups of sporadic groups and their automorphism groups.

$\bullet$ Let $(S,nX)=(McL,2B)$. In this case, 
we have $\operatorname{m}(2B,2B,14A)=14$ and $\operatorname{m}(2B,14A,22A)=16236$. Consequently, there exist three elements in class $2B$ that generate a subgroup $H$ not included in $S$ such that $|H|$ is divisible by $7$ and $11$. But there is no maximal subgroup of $G$ that is not included in $S$ and whose order is divisible by both $7$ and~$11$. 
Therefore, these three elements in $2B$ generate $G$ and so $\alpha^{\phantom{S}}_{S}(x)=3$.
  
$\bullet$ Let $(S,nX)=(He,2C)$.
We have $\operatorname{m}(2C,2C,14CD)=14$. (For $He.2$, class $14CD$ in the notation of \cite{atlas} is class $14B$ in \cite{GAP}.) This, together with Lemma~\ref{GuKa}$(ii)$, implies that there are three elements in $2C$ that generate $G$. Thus, $\alpha^{\phantom{S}}_{S}(x)=3$.

$\bullet$ Let $(S,nX)=(Suz,2C)$.
We have $\operatorname{m}(2C,2C,7A)=7$ and $\operatorname{m}(2C,7A,22A)=3630$. Therefore, there are three elements in~$2C$ that generate a subgroup $H$ which is not included in $S$ and has order divisible by $7$ and $11$.  But the order of every maximal subgroup of $Suz.2$ not included in $S$ is not divisible by either $7$ or $11$. Thus, $H=G$ and $\alpha^{\phantom{S}}_{S}(x)=3$.
 
$\bullet$ Let $(S,nX)=(Suz,2D)$. Then $\operatorname{m}(2D,2D,14A)=14$. By Lemma~\ref{GuKa}$(ii)$, there are three elements in $2C$ that generate $G$. Thus, $\alpha^{\phantom{S}}_{S}(x)=3$.

$\bullet$ Let $(S,nX)=(Fi_{22},2E)$.
Then $\operatorname{m}(2E,2E,16AB)=16$. (For $Fi_{22}.2$, class $16AB$ in \cite{atlas} is class $16A$ in \cite{GAP}.) As in the previous case, $\alpha^{\phantom{S}}_{S}(x)=3$ by Lemma~\ref{GuKa}$(ii)$.

$\bullet$ Let $(S,nX)=(Fi_{22},2F)$.
In this case, we have $\operatorname{m}(2F,2F,11A)=11$ and $\operatorname{m}(2F,11A,42A)=1867488$. Hence, there exist three elements in~$2F$ that generate a subgroup $H$ not included in $S$ such that $|H|$ is divisible by $7$ and $11$ and $H$ contains an element of order~$21$. Every maximal subgroup of $Fi_{22}.2$ of order divisible by both $7$ and $11$ not included in $S$ is isomorphic to either $2\,.\, U_6(2)\,.\, 2$ or ${2^{10}.\,M_{22}\,.\,2}$. But these groups contain no elements of order~$21$. Thus, $H=G$ and $\alpha^{\phantom{S}}_{S}(x)=3$.

$\bullet$ Let $(S,nX)=({Fi_{24}}',2D)$.
In this case, we have $\operatorname{m}(2D,2D,33A)=33$ and $\operatorname{m}(2D,33A,46A)=172322171820$. Thus, there are three elements in~$2D$ that generate a subgroup $H$ not being included in $S$, having order divisible by $11$ and $23$, and containing an element of order~$33$. Every maximal subgroup of~$G$ of order divisible by both $11$ and $23$ that is not included in $S$ is isomorphic to either $Fi_{23}\times 2$ or ${2^{12}.M_{24}}$, see \cite[Theorem~1.1 and Table~1.1]{LW}. But these groups contain no elements of order~$33$. Again, $H=G$ and $\alpha^{\phantom{S}}_{S}(x)=3$.

\medskip

The proof of $(i)$ is complete.

\medskip

$\bullet$ Let $(S,nX)$ be one of $(M_{22},2B)$, $(HS,2C)$, or $(Fi_{22},2D)$. In this case, as we mentioned in the introduction, $G=\operatorname{Aut} S$ is a $4$-transposition group with respect to the class $D=nX$. This can be verified by checking that $\operatorname{m}(x,x,y)\ne 0$ only when $y$ has order at most $4$. By \cite[Theorem 2]{RZ2}, we know that $3 \leqslant \alpha^{\phantom{S}}_S(x)\leqslant 4$. However, $3$-generated $4$-transposition groups have been classified, see \cite[Sections 3.1, 3.2]{KMS} and \cite[Theorem 3.4.2]{K}. This classification implies, in particular, that $S$ cannot be a composition factor of a $3$-generated $4$-transposition group as the only possible such factor is $L_2(7)$. This shows that $\alpha^{\phantom{S}}_S(x)=4$.

In the case $(S,nX)=(HS,2C)$, we can also give an alternative proof which does not rely on the classification of $3$-generated $4$-transposition groups, but uses Brauer's trick (Lemma \ref{uf}) instead. Assume to the contrary that $G=\operatorname{Aut} S$ is generated by some distinct $x_1,x_2,x_3\in 2C$. Denote $A=\langle x_1,x_2\rangle$ and $B=\langle x_3 \rangle$. Clearly, $A\cap B=1$. Hence, to obtain a contradiction by Lemma \ref{uf}, it suffices to show that 
\begin{equation}\label{brin}
 ( \chi_A, 1_A ) + ( \chi_B, 1_B ) >  \chi(1)
\end{equation}
for some nonprincipal irreducible ordinary character $\chi$ of $G$. We may take for $\chi$
a characters of degree $22$; namely, the one that takes positive values on the elements in~$2C$.
The required values of $\chi$ are as follows:

$$
\begin{tabular}{c|rrrrrr}
       &  1A & 2A & 2B & 3A & 4B & 2C \\
\hline       
$\chi$   &   22 &  6 & -2 & 4 &  2 & 8 
\end{tabular}
$$

We have
$$
( \chi_B, 1_B ) = \frac{1}{2}\big(\chi(1)+\chi(2C)\big) = \frac{1}{2}(22+8) =15
$$
and so (\ref{brin}) is equivalent to
\begin{equation}\label{brin2}
 ( \chi_A, 1_A ) > 7
\end{equation}

The possible conjugacy classes the product $x_1x_2$ belongs to are  $2A$, $2B$, $3A$, and $4B$, 
which follows from determining when $\operatorname{m}(2C,2C,nX)\ne 0$ for various $nX$.
We consider these cases separately.

If $x_1x_2\in 2A$ then $A\cong \mathbb{Z}_2\times \mathbb{Z}_2$ and 
$$
( \chi_A, 1_A ) = \frac{1}{4}\big(\chi(1)+\chi(2A)+2\chi(2C)\big) =\frac{1}{4}(22+6+2\cdot 8)=11 > 7.
$$

If $x_1x_2\in 2B$ then also $A\cong \mathbb{Z}_2\times \mathbb{Z}_2$ and 
$$
( \chi_A, 1_A ) = \frac{1}{4}\big(\chi(1)+\chi(2B)+2\chi(2C)\big) 
=\frac{1}{4}(22-2+2\cdot 8)= 9 > 7.
$$

If $x_1x_2\in 3A$ then $A\cong S_3$ and 
$$
( \chi_A, 1_A ) = \frac{1}{6}\big(\chi(1)+2\chi(3A)+3\chi(2C)\big) =\frac{1}{6}(22+2\cdot 4+3\cdot 8)=9 > 7.
$$

If $x_1x_2\in 4B$ then $A\cong D_8$. Observe that the squares of elements in $4B$ are in $2A$. Also, both classes of noncentral involutions in $D_8$ fuse to $2C$, because each generator $x_1$ and $x_2$ is in its own class. Therefore, we have
$$
( \chi_A, 1_A ) = \frac{1}{8}\big(\chi(1)+\chi(2A)+2\chi(4B)+4\chi(2C)\big) 
=\frac{1}{8}(22+6+2\cdot 2+4\cdot 8)=8 > 7.
$$
We see that in all cases inequality (\ref{brin2}) holds and so $\langle A,B\rangle$ is a proper subgroup of~$G$ by Lemma \ref{uf}, a contradiction.
\end{proof}

\section{Final remarks}

Both the original results by Guralnick and Saxl and their refinements find wide application. For example, in the proper paper \cite{GS}, the estimates on $\alpha^{\phantom{S}}_{S}(x)$ are used to classify bireflection groups.  Di\,Martino, Pellegrini, and Zalesskii used their results to study almost cyclic matrices in representations of sporadic groups. The recent paper~\cite{PSV} by I.\,N.\,Ponomarenko, S.\,V.\,Skresanov, and A.\,V.\,Va\-sil'ev uses the estimates from \cite{GS} to study the inheritance of properties of permutation groups by their $k$-closures. Finally, another area where the estimates on $\alpha^{\phantom{S}}_{S}(x)$ are used and play a key role is proving analogs of Baer--Suzuki theorems, see  \cite{FGG,Gu,GuestLevy,GGKP,GGKP1,GGKP2,RZ1,LWR,WGR,YWR,YWRV,YRV}. It is often necessary (see, e.\,g., \cite{RZ1,LWR,WGR,YWR,YWRV}) to refine the results of \cite{GS} for certain groups and automorphisms. The present work may be viewed as a contribution to finding such refinements.


\medskip

{\em Acknowledgement.\/}
The authors are thankful to Prof. Andrey\,S.\,Mamontov for helpful consultations.

\medskip

This is a preprint of the paper of the same title accepted for publication in the journal ``Siberian Electronic Mathematical Reports'' \big(© Sobolev Institute of Mathematics of SB RAS, 2025\big).

\end{document}